\documentclass[10pt]{amsart}

\usepackage{amssymb}
\usepackage{dsfont}
\usepackage{amscd}
\usepackage[mathscr]{euscript} 
\usepackage[all]{xy}
\usepackage{url}
\usepackage{stmaryrd}
\usepackage{comment}
\usepackage[retainorgcmds]{IEEEtrantools}
\usepackage{hyperref}
\usepackage{graphicx}
\usepackage{mathtools}
\usepackage[dvipsnames]{xcolor}
\usepackage{centernot}
\usepackage{marvosym}
\usepackage{tikz}
\usepackage{tikz-cd}
\usepackage{xspace}
\usepackage{bbm}
\usepackage[title]{appendix}

\usepackage{tkz-fct}
\usepackage{soul}
\usepackage{lineno}
\usepackage[T1]{fontenc}
\usepackage[utf8]{inputenc}

\usepackage[english]{babel}
\usepackage[autostyle]{csquotes}
\usepackage[shortlabels]{enumitem}

\makeatletter
\def\author@andify{%
  \nxandlist {\unskip ,\penalty-1 \space\ignorespaces}%
    {\unskip {} \@@and~}%
    {\unskip \penalty-2 \space \@@and~}%
}

\title{Measures on $\aut(M)$}

\author[D. M. HOFFMANN]{Daniel Max Hoffmann$^{\dagger}$}
\thanks{$^{\dagger}$SDG. The author is supported by the Narodowe Centrum Nauki grants no. 2021/43/B/ST1/00405.}
\address{$^{\dagger}$Instytut Matematyki\\
Uniwersytet Warszawski\\
Warszawa\\
Poland}
\email{daniel.max.hoffmann@gmail.com}

 \DeclareMathOperator{\aut}{Aut}

\DeclareMathOperator{\Rr}{\mathbb{R}}

\newtheorem{theorem}{Theorem}[section]
\newtheorem{proposition}[theorem]{Proposition}
\newtheorem{lemma}[theorem]{Lemma}
\newtheorem{cor}[theorem]{Corollary}
\newtheorem{fact}[theorem]{Fact}
\theoremstyle{definition}
\newtheorem{definition}[theorem]{Definition}

\newtheorem{remark}[theorem]{Remark}

\theoremstyle{remark}
\newtheorem*{theorem*}{Theorem}
\newtheorem*{cor*}{Corollary}

\theoremstyle{definition}
\theoremstyle{definition}

\theoremstyle{definition}

\theoremstyle{remark}

\makeatletter
\providecommand*{\cupdot}{%
  \mathbin{%
    \mathpalette\@cupdot{}%
  }%
}
\newcommand*{\@cupdot}[2]{%
  \ooalign{%
    $\m@th#1\cup$\cr
    \sbox0{$#1\cup$}%
    \dimen@=\ht0 %
    \sbox0{$\m@th#1\cdot$}%
    \advance\dimen@ by -\ht0 %
    \dimen@=.5\dimen@
    \hidewidth\raise\dimen@\box0\hidewidth
  }%
}

\providecommand*{\bigcupdot}{%
  \mathop{%
    \vphantom{\bigcup}%
    \mathpalette\@bigcupdot{}%
  }%
}
\newcommand*{\@bigcupdot}[2]{%
  \ooalign{%
    $\m@th#1\bigcup$\cr
    \sbox0{$#1\bigcup$}%
    \dimen@=\ht0 %
    \advance\dimen@ by -\dp0 %
    \sbox0{\scalebox{2}{$\m@th#1\cdot$}}%
    \advance\dimen@ by -\ht0 %
    \dimen@=.5\dimen@
    \hidewidth\raise\dimen@\box0\hidewidth
  }%
}
\makeatother

\def\Ind#1#2{#1\setbox0=\hbox{$#1x$}\kern\wd0\hbox to 0pt{\hss$#1\mid$\hss}
\lower.9\ht0\hbox to 0pt{\hss$#1\smile$\hss}\kern\wd0}

\def\notind#1#2{#1\setbox0=\hbox{$#1x$}\kern\wd0
\hbox to 0pt{\mathchardef\nn=12854\hss$#1\nn$\kern1.4\wd0\hss}
\hbox to 0pt{\hss$#1\mid$\hss}\lower.9\ht0 \hbox to 0pt{\hss$#1\smile$\hss}\kern\wd0}

\def\Ind#1#2{#1\setbox0=\hbox{$#1x$}\kern\wd0\hbox to 0pt{\hss$#1\mid$\hss}
\lower.9\ht0\hbox to 0pt{\hss$#1\smile$\hss}\kern\wd0}

\def\notind#1#2{#1\setbox0=\hbox{$#1x$}\kern\wd0
\hbox to 0pt{\mathchardef\nn=12854\hss$#1\nn$\kern1.4\wd0\hss}
\hbox to 0pt{\hss$#1\mid$\hss}\lower.9\ht0 \hbox to 0pt{\hss$#1\smile$\hss}\kern\wd0}

\newcommand{\CM}{{\mathcal M}}

\newcommand{\conv}{\mathrm{conv}}
\newcommand{\Av}{\mathrm{Av}}

\begin{document}

\maketitle
\vspace{-8mm}
\begin{center}
{\small Instytut Matematyki, Uniwersytet Warszawski}
\end{center}

\section{Introduction}
If a topological group $G$ acts on a compact Hausdorff space $X$, mathematicians often seek to understand the action of the space of measures $\CM(G)$ on the space of measures $\CM(X)$. To explain this, consider measures $\mu \in \CM(G)$ and $\nu \in \CM(X)$. The action is defined as follows:
\begin{equation*}
    (\mu * \nu)(A) = \int\limits_{g\in G} g_{*}\nu(A) \, d\mu(g),
\end{equation*}
where $g_{*}\nu$ denotes the push-forward of $\nu$ under the mapping $a \mapsto g \cdot a$, for $a \in X$. Specifically, we have $g_{*}\nu(A) = \nu(g^{-1}(A)) = (g \cdot \nu)(A)$ for the induced action of $G$ on $\CM(X)$, where $(g \cdot \nu) = \nu \circ g^{-1}$. 
If the action of $G$ on $X$ is tame, it should follow that the action of $\mathcal{M}(G)$ on $\mathcal{M}(X)$ is also tame. We can investigate this problem in the model theoretic set-up and therefore we need to describe a measure-theoretic framework suitable for us.

Unfortunately, many sources focus only on locally compact groups or Polish groups and since $\aut(M)$ might be not such a group, we need a more general approach. For example, if $G$ is a topological group and $\CM(G)$ is the space of regular Borel probability measures on $G$, say $\mu,\nu\in\CM(G)$, we want to define the \emph{convolution} $\mu\ast\nu$ and then equip $\CM(G)$ with the structure of a semigroup by using it. The standard approach is as on page 15 in \cite{Kallenberg}, but as we already mentioned, it does not work as the multiplication map $m:G\times G\to G$ might be not measurable in $\mathcal{B}(G)\otimes\mathcal{B}(G)$ ($\neq \mathcal{B}(G\times G)$ - this is not the case for separable metric groups, see Lemma 1.2 in \cite{Kallenberg}). We decided to investigate both cases: when $G$ has nice topological properties (e.g. $\aut(M)$ for countable model $M$) and when we have no control over the topological properties of $G$. For the second case, we need to choose a proper space of measures to be able to convolute measures.

\section{$\tau$-additive and Keisler measures}
Model theory usually works with regular Borel probability measures, which on the space of types turn out to be the Keisler measures (cf. Chapter 7 in \cite{Guide_NIP}). By studying \cite{Fremlin}, we found other class of measures which corresponds to Keilser measures, but behaves well under the convolution operation. As \cite{Fremlin} is a very comprehensive source, we decided to recall some of the definitions and results from it here, so the reader will not need to search for them. 

Let $\mathcal{T}$ be a topology on $X$, and let $(X,\Sigma,\mu)$ be \emph{measure space} (i.e. $\mu$ takes values in $[0,\infty]$, $\mu(\emptyset)=0$ and $\mu$ is countably additive on disjoint countable families of sets). 
We say that $(X,\mathcal{T},\Sigma,\mu)$ is a \emph{topological measure space} if $\mathcal{T}\subseteq\mathcal{B}(X)\subseteq\Sigma$.

\begin{definition}[411H in \cite{Fremlin}]
    A quasi-Radon measure space is a topological space $(X,\mathcal{T},\Sigma,\mu)$ such that
    \begin{enumerate}
        \item $(X,\Sigma,\mu)$ is complete,
        \item $(X,\Sigma,\mu)$ is locally determined,
        \item $\mu$ is $\tau$-additive,
        \item $\mu$ is inner regular with respect to the closed sets,
        \item $\mu$ is effectively locally finite.
    \end{enumerate}
\end{definition}

The properties \emph{locally determined} and \emph{effectively locally finite} follow automatically if the measure $\mu$ is a probability measure, which will be our main focus here. The assumption about being a complete measure is quite harmless, as we can always pass to the completion of a given measure (cf. 212C in \cite{Fremlin}) which preserves outer regularity (with respect to open sets), inner regularity (with respect to closed/compact sets), $\tau$-additivity etc.
The key property here is the \emph{$\tau$-additivity}:

\begin{definition}[411C in \cite{Fremlin}]\label{def.tau.add}
    Let $(X,\Sigma,\mu)$ be a measure space and $\mathcal{T}$ a topology on $X$. We say that $\mu$ is $\tau$-additive if whenever $\mathcal{G}$ is a non-empty upwards directed family of open sets such that $\mathcal{G}\subseteq\Sigma$ and
    $\bigcup\mathcal{G}\in\Sigma$ then 
    $$\mu(\bigcup\mathcal{G})\sup\limits_{G\in\mathcal{G}}\mu(G).$$
\end{definition}

Because $\tau$-additivity is not a frequently used notion, we start with several facts characterizing it and then we show that in the case of a space of types we recover the Keisler measures, as expected.
For the first example, 
note that for any topological space $X$, the Dirac measures (considered on $\Sigma=\mathcal{P}(X)$) are quasi-Radon probabilitic measures. Similarly for finite convex combinations of Dirac measures.

\begin{fact}[411E in \cite{Fremlin}]\label{Fre.411E}
    Let $(X,\Sigma,\mu)$ be a measure space, where $X$ is a topological space.
    If $\mu$ is inner regular with respect to the compact sets then $\mu$ is $\tau$-additive.
\end{fact}

\noindent
By the above fact, we see that if $\mu$ is a Borel probability measure on a compact space $X$ such that $\mu$ is inner regular with respect to the closed sets, then the completion of $\mu$ is a quasi-Radon measure. Not all topological spaces which we want to study are compact, thus we provide the next fact for the regular spaces.

\begin{fact}[415C in \cite{Fremlin}]\label{Fre.415C}
Let $X$ be a regular topological space.
\begin{enumerate}
    \item If $\mu$ is a complete, locally determined, effectively locally finite,
    $\tau$-additive topological measure on $X$,
    inner regular with respect to the Borel sets, then it is a quasi-Radon measure.

    \item If $\mu$ is an effectively locally finite, $\tau$-additive Borel measure on $X$, then its c.l.d. version is a quasi-Radon measure. In particular, if $\mu$ is a $\tau$-additive Borel probability measure on $X$ then its completion $\hat{\mu}$ (equal to its c.l.d. version) is a quasi-Radon measure.
\end{enumerate}
\end{fact}

\noindent
For a moment, let us work with $X$ being a space of types, so in particular a compact regular space.
Then the situation can be illustrated as follows (regular means inner regular with respect to the compact sets and outer regular with respect to the open sets):
\begin{equation*}
\left\{
\setlength\arraycolsep{0pt}
\begin{array}{ c   >{{}}c<{{}} c    @{{}={}} l}
\text{regular} \\
\text{Borel probability} \\
\text{measures on }X
\end{array}
\right\} \;=\; \left\{
\setlength\arraycolsep{0pt}
\begin{array}{ c   >{{}}c<{{}} c    @{{}={}} l}
\tau\text{-additive} \\
\text{Borel probability} \\
\text{measures on }X
\end{array}
\right\} \;\xrightarrow[]{\text{completion}}\; \left\{
\setlength\arraycolsep{0pt}
\begin{array}{ c   >{{}}c<{{}} c    @{{}={}} l}
\text{probability} \\
\text{quasi-Radon} \\
\text{measures on }X
\end{array}
\right\}
\end{equation*}
For left-to-right inclusion in the equality above, we argue as follows. Since $X$ is compact and Hausdorff, closed sets and compact sets coincide. Thus inner regularity with respect to the closed sets is indeed inner equality with repsect to the compact sets and we obtain $\tau$-additivity by Fact \ref{Fre.411E}. For the right-to-left inclusion, point (b) from Fact \ref{Fre.415C} says that a completion $\hat{\mu}$ of a $\tau$-additive Borel probability measure on $X$ is a quasi-Radon measure (i.e. the arrow to the last box above), thus $\mu$ is inner regular with respect to the closed sets. We only need to show the outer regularity with respect to the open sets, which follows by:

\begin{fact}[412Wb in \cite{Fremlin}]\label{Fre.412Wb}
    Let $(X,\Sigma,\mu)$ be a measure space, where $X$ is a topological space.
    Assume that there is a countable open cover of $X$ by sets of finite measure.
    The following are equivalent.
    \begin{enumerate}
        \item $\mu$ is inner regular with respect to the closed sets.
        \item $\mu$ is outer regular with respect to the open sets,
        \item for any measurable set $E\subseteq X$ and $\epsilon>0$, 
        there are a measurable closed set $F\subseteq E$ and a measurable open set $H\supseteq E$
        such that $\mu(H\setminus F)<\epsilon$.
    \end{enumerate}
\end{fact}

From the above, we see that in the case of the space of types, we get back our old friends - i.e. Keisler measures.
However, let us finally argue that in the case of Polish spaces (and Polish groups, which come to our picture as well), regular Borel probability measures coincide with $\tau$-additive Borel probability measures, and completion of any of the latter ones is a probability quasi-Radon measure.

\begin{fact}[415D in \cite{Fremlin}]\label{Fre.415D}
Let $X$ be a hereditary Lindel\"{o}f topological space (e.g. separable metrizable). 
If $\mu$ is an effectively locally finite Borel measure on $X$, then its completion is a quasi-Radon measure.
In particular, if $\mu$ is a Borel probability measure on $X$ then its completion $\hat{\mu}$ is a quasi-Radon probability measure.
\end{fact}

Thus, if $X$ is a Polish space and $\mu$ is a (regular or not...) Borel probability measure on $X$ then its completion $\hat{\mu}$ is a quasi-Radon measure, hence $\mu$ is $\tau$-additive. 
Moreover, as every Polish space is a Radon space, all Borel probability measures are inner regular with respect to the compact sets, and so outer regular with respect to the open sets by Fact \ref{Fre.412Wb}.

After all of that, we are quite convinced to start to work with $\tau$-additive Borel probability measures and so, if $X$ is a topological space, let $\CM_\tau(X)$ denote the set of all such measures over $X$.

\section{Convolution action for $\tau$-additive measures}
Coming back to the case of a topological group $G$ acting on a topological space $X$, so to a dynamical system $(X,G)$,
we are interested in obtaining a discrete dynamical system $(\mathcal{M}_\tau(X),\mathcal{M}_\tau(G))$.
The following fact is one of the main properties of $\tau$-additive measures, which allows us to define the so-called convolution action.

\begin{fact}[417B in \cite{Fremlin}]\label{Fre.417B}
    Let $X_1,X_2$ be topological spaces and let $\nu$ be a $\tau$-additive topological measure on $X_2$.
    For simplicity, assume also that $\nu$ is a probability measure.
    \begin{enumerate}
        \item If $W\subseteq X_1\times X_2$ is Borel then the map
        $$X_1\ni x\mapsto \nu\big(W[\{x\}] \big)\in[0,\infty]$$
        is Borel measurable.

        \item For any $f\in C_b(X_1\times X_2)$, the map
        $$X_1\ni x_1\mapsto\int f(x_1,x_2)d\nu(x_2)$$
        is continuous.
    \end{enumerate}
\end{fact}

Now, we could follow \cite{Fremlin} to derive 417C and use it in 444A to define the convolution of two totally finite quasi-Radon measures on a topological group, and then note associativity (444B) and state Fubini's formula for that convolution (444C). Perhaps, we could also note that the procedure works well for all $\tau$-additive totally finite Borel measures on a topological group (444E in \cite{Fremlin}). The disadvantage of that method is that we do not cover the situation of an action of a topological group $G$ on a topological space $X$, i.e. we would still need to define the convolution of a $\tau$-additive measure on $G$ with a $\tau$-additive measure on $X$, and then check some properties. Therefore, we shortly define the last type of convolution, which coincides with the convolution from 444A/444E in \cite{Fremlin} for the action of $G$ on itself. We start with bringing a result, on which everything is based.

\begin{fact}[417D in \cite{Fremlin}]\label{Fre.417D}
Let $\big((X_i,\mathcal{T}_i,\Sigma_i,\mu_i)\big)_{i\in I}$
be a finite family of effectively locally finite, $\tau$-additive topological measure spaces.
Then there is a unique complete, locally determined, effectively locally finite, $\tau$-additive
topological measure $\Tilde{\lambda}$ on $X=\prod\limits_{i\in I}X_i$, 
inner regular with respect to the Borel sets, such that 
$$\Tilde{\lambda}\Big(\prod\limits_{i\in I}A_i \Big)=\prod\limits_{i\in I}\mu_i(A_i)$$
whenever $A_i$ is open subset of $X_i$, where $i\in I$. The measure $\Tilde{\lambda}$ will be denoted $\prod\limits_{i\in I}\mu_i$.
\end{fact}

\begin{lemma}[Convolution]\label{lemma.convolution}
    Assume that a topological group $G$ acts on a topological space $X$ via $\pi:G\times X\to X$.
    Let $\mu\in\CM_{\tau}(G)$ and let $\nu\in\CM_{\tau}(X)$.
    Then the formula
    $$(\mu\ast\nu)(E):=(\mu\times\nu)\big( \pi^{-1}[E]\big),$$
    where $E\subseteq X$ is Borel, defines a $\tau$-additive Borel probability measure on $X$.
\end{lemma}

\begin{proof}
    As $\mu\ast\nu$ is simply a push-forward of the measure $\mu\times\nu$ from Fact\ref{Fre.417D} it is a well-defined (Borel) measure. By the formula from Fact \ref{Fre.417D}, we have
    $$(\mu\ast\nu)(X)=(\mu\times\nu)(G\times X)=\mu(G)\cdot\nu(X)=1.$$
    Finally, the $\tau$-addivity follows easily by definitions.
\end{proof}

    In the assumption of Lemma \ref{lemma.convolution}, we have moreoever:
    \begin{enumerate}
        \item $$(\mu\ast\nu)(E)=\int (g\cdot\nu)(E) d\mu(g)=\int \nu(g^{-1}E)d\mu(g),$$
        \item and if $X$ is regular then 
        $$(\mu\ast\nu)(E)=\int \mu(E:x)d\nu(x),$$
    \end{enumerate}
    where $E:x:=\{g\in G\;|\; gx\in E\}$.
    
    To see the first point, we pass to completions $\hat{\mu}$ and $\hat{\nu}$ of $\mu$ and $\nu$ respectively.
    Then we use Theorem 417Ha in \cite{Fremlin} to note that
    \begin{IEEEeqnarray*}{rCl}
    (\mu\ast\nu)(E) &=& (\hat{\mu}\times\hat{\nu})\big(\pi^{-1}[E]\big) = \int (\mathbbm{1}_E\circ\pi)(g,x)d(\hat{\mu}\times\hat{\nu}) \\
    &=& \int \int (\mathbbm{1}_E\circ\pi)(g,x)d\hat{\nu}(x)\,d\hat{\mu}(g) = \int \Big( \int \mathbbm{1}_{g^{-1}E}(x)d\hat{\nu}(x)\Big)d\hat{\mu}(g) \\
    &=& \int \hat{\nu}(g^{-1}E)d\hat{\mu}(g)=\int \nu(g^{-1}E)d\hat{\mu}(g).
    \end{IEEEeqnarray*}
    By Proposition 212Fb in \cite{Fremlin}, the last integral is equal to
    $$\int \nu(g^{-1}E)d\mu(g) = \int (g\cdot \nu)(E)d\mu(g)$$
    (as $g\cdot\nu=\nu\circ g^{-1}$). 
    Note that by Corollary \ref{cor.444F}, the function in the last integral is Borel measurable, provided $X$ is regular.

    For the second point, we need a refined (in the case of probability measures) version of Theorem 417C from \cite{Fremlin} and its point (iv).

\begin{remark}\label{rem:refined.417C}
    Let $X$ and $Y$ be regular topological spaces, let $\mu\in\CM_{\tau}(X)$ and $\nu\in\CM_{\tau}(Y)$, 
    and let $\Tilde{\lambda}=\mu\times\nu$ be the product measure as in Fact \ref{Fre.417D}.
    Then for every Borel set $W\subset X\times Y$ we have
    $$(\mu\times\nu)(W)=\int \nu(W_x)d\mu(x) = \int \mu(W^y)d\nu(y),$$
    where $W_x:=\{y\in Y\;|\; (x,y)\in W\}$ and $W^y:=\{x\in X\;|\;(x,y)\in W\}$.
\end{remark}

\begin{proof}
    By point (iv) of Theorem 417C from \cite{Fremlin}, our remark holds in the case of $W$ being an open set.
    Then, then complement $W^c$ is closed and we have
    \begin{IEEEeqnarray*}{rCl}
    (\mu\times\nu)(W^c) &=& 1 - (\mu\times\nu)(W) = 1-\int \nu(W_x)d\mu(x) \\
    &=& \int \Big( 1-\nu(W_x) \Big)d\mu(x) = \int \nu(W^c_x)d\mu(x),
    \end{IEEEeqnarray*}
    and similarly $(\mu\times\nu)(W^c)=\int\mu\big((W^c)^y\big)d\nu(y)$. Thus our remark holds for both the open and the closed sets.

    Since $X\times Y$ is regular, Fact \ref{Fre.415C} implies that the completion of $\mu\times\nu$ is a quasi-Radon measure.
    Hence, $\mu\times\nu$ is inner regular with respect to the closed sets. Let $W\subseteq X\times Y$ be a Borel set now and let $\epsilon>0$.
    By Fact \ref{Fre.412Wb}, there exist a closed subset $F_{\epsilon}\subseteq W$ and an open subset $H_{\epsilon}\supseteq W$ such that $(\mu\times\nu)(H_{\epsilon}\setminus F_{\epsilon})<\epsilon$.
    As $(F_{\epsilon})_x\subseteq W_x\subseteq (H_{\epsilon})_x$, we have
    $$(\mu\times\nu)(W)-\epsilon\leqslant(\mu\times\nu)(F_{\epsilon})\leqslant (\mu\times\nu)(W)\leqslant (\mu\times\nu)(H_{\epsilon})\leqslant(\mu\times\nu)(W)+\epsilon.$$
    Again, as $(F_{\epsilon})_x\subseteq W_x\subseteq (H_{\epsilon})_x$, we have
    $$\int \nu\big( (F_{\epsilon})_x\big)d\mu(x)\leqslant \int \nu\big(W_x)d\mu(x) \leqslant \int \nu\big( (H_{\epsilon})_x\big)d\mu(x).$$
    By what we proved up to this moment, we conclude that for every $\epsilon>0$ it is
    $$(\mu\times\nu)(W)-\epsilon\leqslant\int \nu(W_x)d\mu(x)\leqslant (\mu\times\nu)(W)+\epsilon,$$
    and we get that $(\mu\times\nu)(W)=\int\nu(W_x)d\mu(x)$ for any Borel set $W\subseteq X\times Y$.
    The other equality if proved in a smillar manner.
\end{proof}

\section{Convolution semigroup}
Coming back to the action of $G$ on $X$, which is a regular space, we see that
$$(\mu\ast\nu)(E)=\int \mu(E:x)d\nu(x),$$
where $E\subseteq X$ is a Borel set, follows directly by Remark \ref{rem:refined.417C}.

\begin{lemma}\label{lemma:convolution.functions}
    Assume that a topological group $G$ acts on a regular topological space $X$ via $\pi:G\times X\to X$.
    Let $\mu\in\CM_{\tau}(G)$, $\nu\in\CM_{\tau}(X)$.
    Then for any $(\mu \ast \nu)$-integrable $f:X\to\Rr$, we have
    \begin{IEEEeqnarray*}{rCl}
    \int fd(\mu\ast\nu) &=& \int f(g\cdot x)d(\mu\times\nu)(g,x)=\int\int f(g\cdot x)d\mu(g)\,d\nu(x) \\
        &=& \int\int f(g\cdot x)d\nu(x)\,d\mu(g).
    \end{IEEEeqnarray*}
\end{lemma}

\begin{proof}
    The proof of Proposition 444C from \cite{Fremlin} can be repeated here.
\end{proof}

To show associativity of the convolution of measures, we need one more fact:

\begin{fact}[444F in \cite{Fremlin}]\label{Fre.444F}
    Let $G$ be a topological group acting on a topological space $X$, and let $\nu$ be a $\sigma$-finite quasi-Radon measure on $X$.
    \begin{enumerate}
        \item $G\ni g\mapsto \nu(gE)\in[0,\infty]$ is Borel measureable for any Borel $E\subseteq X$.
    
        \item If $f:X\to\Rr$ is Borel measurable, then 
        $$Q =\{g\in G\;|\;\int g\cdot f(x)d\nu(x) \text{ is defined in }[-\infty,\infty]\},$$ 
        where $g\cdot f(x):=f(g^{-1}\cdot x)$, is a Borel set, and 
        $$Q\ni g\mapsto\int g\cdot f(x)d\nu(x)\in[-\infty,\infty]$$
        is Borel measurable
    \end{enumerate}
\end{fact}

\begin{cor}\label{cor.444F}
    Let $G$ be a topological group acting on a regular topological space $X$, and let $\nu$ be a $\tau$-additive Borel probability measure on $X$. Then 
    $$G\ni g\mapsto \nu(gE)=\hat{\nu}(gE)$$
    is Borel measureable for any Borel $E\subseteq X$.
\end{cor}

\begin{lemma}
     Assume that a topological group $G$ acts on a regular topological space $X$ via $\pi:G\times X\to X$.
    Let $\mu_1,\mu_2\in\CM_{\tau}(G)$ and let $\nu\in\CM_{\tau}(X)$.
    We have
    $$(\mu_1\ast\mu_2)\ast\nu=\mu_1\ast(\mu_2\ast\nu).$$
    Moreover $\delta_1\ast\nu=\nu$ for the Dirac measure $\delta_1$ on the neutral element $1$ of group $G$.
\end{lemma}

\begin{proof}
    Let $\pi:G\times X\to X$ denote the action of $G$ on $X$, and let $E\subseteq X$ be a Borel set.
    Now, by applying Lemma \ref{lemma:convolution.functions} and previously derived formulas, we get
    \begin{IEEEeqnarray*}{rCl}
    \big( (\mu_1\ast\mu_2)\ast\nu\big)(E) &=& \Big( (\mu_1\ast\mu_2)\times \nu\Big)\big(\pi^{-1}[E]\big) \\
    &=& \int \nu(g^{-1}E)d(\mu_1\ast\mu_2)\qquad \Big(g\mapsto\nu(g^{-1}E)\text{ is $\mu_1\ast\mu_2$-meas. by Cor. \ref{cor.444F}}\Big) \\
    &=& \int \int \nu\big( (gh)^{-1}E\big)d\mu_1(g)d\mu_2(h) \\
    &=& \int \int \nu\big( h^{-1}(g^{-1}E)\big)d\mu_1(g)d\mu_2(h) \\
    &=& \int \Big(\int \nu\big( h^{-1}(g^{-1}E)\big)d\mu_2(h)\Big) d\mu_1(g) \\
    &=& \int \Big( (\mu_2\ast\nu)(g^{-1}E) \Big)d\mu_1(g) \\
    &=& \big(\mu_1\ast(\mu_2\ast\nu)\big)(E).
    \end{IEEEeqnarray*}
    The proof of the moreover part is straightforward.
\end{proof}

\begin{cor}\label{cor:A15}
    If $G$ is a topological group, the set $\CM_{\tau}(G)$ equipped with the convolution action on itself forms a semi-group. Moreover, in the case of a Polish group $G$, the set $\CM_\tau(G)$ with the convolution and equipped with the weak topology (see below) forms a Polish semigorup (cf. Proposition 4.3 in \cite{Yano22}).
\end{cor}

Now, we would like to equip $\CM_{\tau}(X)$ with a topology, and then show that $(\CM_{\tau}(X),\CM_{\tau}(G))$
is a discrete dynamical system.
We simply follow the standard approach here, so if $X$ is a topological space then we equip 
$\CM_{\tau}(X)$ with the weak topology, where the basic open sets are of the form
$$\bigcap\limits_{i\leqslant n}\{\nu\in\CM_\tau(X)\;|\;r_i<\int f_id\nu< s_i\},$$
where $f_i\in C_b(X)$ and $r_i,s_i\in\Rr$ for $i\leqslant n$.
Let us note, that $X\ni x\mapsto \delta_x\in\CM_\tau(X)$ is an embedding of topological spaces,
and $X$ is homeomorphic with its image provided $X$ is normal (e.g. when $X$ is the space of types $S_x(M)$).
By $\conv(X)$ we denote the set of finite convex combinations of Dirac measures $\delta_x$ with $x\in X$.
Let us note that if $g\in G$ and $\nu\in\CM_\tau(X)$, then
$\delta_g\ast\nu=g_\ast\nu$. In particular, if $g\in G$ and $x\in X$ then $\delta_g\ast\delta_x=\delta_{g\cdot x}$.

We need to show that the map
$$\pi_\CM:\CM_\tau(G)\times\CM_\tau(X)\to\CM_\tau(X),$$
given by $(\mu,\nu)\mapsto \mu\ast\nu$ is continuous in the second coordinate (i.e. it is jointly-continuous/continuous if we consider $\CM_\tau(G)$ with the discrete topology).

\begin{remark}\label{rem:continuity.piM}
    The above map $\pi_\CM$ is continuous in the second coordinate. To see it, we assume that
    $\mu\ast\nu$ belongs to some open basic set, i.e.:
    $$\mu\ast\nu\in \bigcap\limits_{i\leqslant n}\{\lambda\in\CM_\tau(X)\;|\;r_i<\int f_id\nu< s_i\}.$$
    Now, define $h_i:X\times G\to \Rr$ by $h_i(x,g):=f_i(g\cdot x)$, where $i\leqslant n$.
    As $f_i\in C_b(X)$ and the map $\pi:G\times X\to X$ is jointly-continuous, we have that each $h_i\in C_b(X\times G)$. Fact \ref{Fre.417B} implies that the map
    $$H_i:X\ni x\mapsto \int h_i(x,g)d\mu(g)\;=\; \int f_i(g\cdot x)d\mu(g)$$
    is continuous, thus $H_i\in C_b(X)$.
    Fix some $i\leqslant n$ and not the that, by Lemma \ref{lemma:convolution.functions},
    $$(r_i,s_i)\ni \int f_id(\mu\ast\nu)= \int H_i d\nu.$$
    Similarly for any $\nu'\in\CM_\tau(X)$, if 
    $$\nu'\in \bigcap\limits_{i\leqslant n}\{\lambda\in\CM_\tau(X)\;|\;r_i<\int H_id\nu< s_i\}$$
    then 
    $$\int f_i d(\mu\ast\nu')=\int H_i d\nu'\in (r_i,s_i).$$
\end{remark}

\begin{cor}\label{cor:discrete.system}
     Let $G$ be a topological group acting on a regular topological space $X$.
     Then $(\CM_\tau(X),\CM_\tau(G))$ is a discrete dynamical system (perhaps non-compact).
\end{cor}

\begin{proposition}\label{prop: equal Ellis2}
Let $G$ be a topological group acting on a regular topological space $X$.
Assume that $\CM_\tau(X)$ is compact Hausdorff.
Then
$E(\CM_\tau(X),\conv(G)) = E(\CM_\tau(X),\CM_\tau(G))$. 
\end{proposition}

\begin{proof} 
The fact that $E(\CM_\tau(X),\conv(G)) \subseteq E(\CM_\tau(X),\CM_\tau(G))$ follows from $\conv(G)\subseteq \CM_\tau(G)$ and definitions.

Now suppose that $\eta \in E(\CM_\tau(X),\CM_\tau(G))$. 
It suffices to show that for any open neighborhood $O\subseteq \CM_\tau(X)^{\CM_\tau(X)}$
of $\eta$, there exists $\bar{g}=(g_1,\ldots,g_n)\in G^n$ such that $(\Av(\bar{g})\ast-)\in O$.
So fix an open neighborhood $O$ of $\eta$, without loss of generality $O$ looks as follows:
\begin{equation*}
    O = \bigcap_{i=1}^{n} \{\beta\in\CM_\tau(X)^{\CM_\tau(X)} \;|\;
    r_i < \int f_i d\beta(\nu_i) < s_i\},
\end{equation*}
for some $f_i\in C_b(X)$ and $r_i,s_i\in\Rr$ for each $i\leqslant n$.

Since $\eta \in E(\CM_\tau(X),\CM_\tau(G))$, there exists some $\mu \in \CM_\tau(G)$ 
such that $(\mu\ast-)\in O$. 
Note that, similarly as in the proof of Remark \ref{rem:continuity.piM}, the maps
$$H_i:G\ni g\mapsto H_i(g)=\int f_i(g\cdot x)d\nu_i(x)\in\Rr$$
are continuous (by Fact \ref{Fre.417B}) and bounded. By Lemma \ref{lemma:convolution.functions}, we compute
$$(r_i,s_i)\ni \int f_id(\mu\ast\nu_i)=\int H_i(g)d\mu(g).$$
Consider $\epsilon>0$ such that for every $i\leqslant n$,
$r_i< \int f_id(\mu\ast\nu_i)-\epsilon$ and $\int f_id(\mu\ast\nu_i)+\epsilon<s_i$.
By 
the Weak Law of Large Numbers
there exists $\bar{g}=(g_1,\ldots,g_n)\in G^n$ such that for every $i\leqslant n$
we have 
$$\int H_i d\mu \approx_{\epsilon} \int H_i d\Av(\bar{g}).$$
Then
$$(r_i,s_i)\ni \int H_id\Av(\bar{g})=\int f_id(\Av(\bar{g})\ast\nu_i),$$
for every $i\leqslant n$ and so $(\Av(\bar{g})\ast-)\in O$.
\end{proof}

\begin{cor}
For every $T$ and every $M\models T$, we have
$$E(\mathfrak{M}_x(M),\conv(\aut(M)))\cong E(\mathfrak{M}_x(M),\CM_\tau(\aut(M))).$$
\end{cor}

We can conclude that for the purpose of doing topological dynamics in model theory, there is no need to go outside of the action of the semigroup $\conv(\aut(M))$.

\bibliographystyle{plain}
\bibliography{biblio2}

\begin{thebibliography}{1}

\bibitem{Fremlin}
D.H. Fremlin.
\newblock {\em Measure Theory, vol. 1-5}.
\newblock Broad Fundations. Torres Fremlin, 2010.

\bibitem{Kallenberg}
Olav Kallenberg.
\newblock {\em Probability and Its Applications}.
\newblock Foundations of Modern Probability. Springer, 2001.

\bibitem{Guide_NIP}
Pierre Simon.
\newblock {\em A Guide to NIP Theories}.
\newblock Lecture Notes in Logic. Cambridge University Press, 2015.

\bibitem{Yano22}
Kouji Yano.
\newblock {Infinite convolutions of probability measures on Polish semigroups}.
\newblock {\em Probability Surveys}, 19:129 -- 159, 2022.

\end{thebibliography}

\end{document}